\documentclass[11pt]{article}
\usepackage{fullpage}
\usepackage[mathscr]{eucal}
\usepackage{epsfig,epsf,psfrag}
\usepackage{amssymb,amsfonts,latexsym}
\usepackage{amsmath}
\usepackage{graphicx}
\usepackage{epstopdf}
\usepackage{bm,xcolor,url}
\usepackage{fixltx2e}%ordering of single and double column floats
\usepackage{array}%array and tabular environments
\usepackage{verbatim}
\usepackage{algpseudocode, cite}
\usepackage{algorithm}
\usepackage{verbatim}
\usepackage{textcomp}
\usepackage{mathrsfs}
\usepackage[referable]{threeparttablex}
\usepackage{subfig}
\usepackage{amsthm}
\usepackage{enumerate}
\usepackage{footnote}
\usepackage{enumitem}
\usepackage{xr}
\usepackage{mathtools}
\catcode`~=11 \def\UrlSpecials{\do\~{\kern -.15em\lower .7ex\hbox{~}\kern .04em}} \catcode`~=13 

\allowdisplaybreaks[3]

\newcommand{\tnorm}[1]{{\left\vert\kern-0.25ex\left\vert\kern-0.25ex\left\vert #1 
    \right\vert\kern-0.25ex\right\vert\kern-0.25ex\right\vert}}
\newcommand{\tnormt}[1]{{\vert\kern-0.25ex\vert\kern-0.25ex\vert #1 
    \vert\kern-0.25ex\vert\kern-0.25ex\vert}}

\newcommand{\norm}[1]{\left\Vert#1\right\Vert}
\newcommand{\normt}[1]{\Vert#1\Vert}
\newcommand{\abst}[1]{\vert#1\vert}

\newcommand{\nn}{\nonumber}

\newcommand{\nt}{\addtocounter{equation}{1}\tag{\theequation}} 

\newcommand{\dom}{\mathsf{dom}\,}

\newcommand{\bareps}{\bar{\varepsilon}}

\newcommand{\ri}{\mathsf{ri}\,}

\newcommand{\dist}{\mathsf{dist}\,}

\newcommand{\calI}{\mathcal{I}}

\newcommand{\calL}{\mathcal{L}}

\newcommand{\calX}{\mathcal{X}}
\newcommand{\calY}{\mathcal{Y}}

\newcommand{\rmH}{\mathrm{H}}

\newcommand{\bbR}{\mathbb{R}}

\DeclareMathAlphabet{\mathbsf}{OT1}{cmss}{bx}{n}
\newcommand{\hatx}{\hat{x}}

\newcommand{\barf}{\bar{f}}

\newcommand{\bart}{\bar{t}}

\newcommand{\barx}{\bar{x}}

\newcommand{\ipt}[2]{\langle{#1},{#2}\rangle}

\newcommand{\lranglet}[2]{\langle{#1},{#2}\rangle}

\newcommand{\lea}{\stackrel{\rm(a)}{\le}}
\newcommand{\leb}{\stackrel{\rm(b)}{\le}}

\DeclareMathOperator*{\argmax}{arg\,max}

\newtheorem{theorem}{Theorem} 
\newtheorem*{theorem*}{Theorem}
\newtheorem{lemma}{Lemma}

\newtheorem{prop}{Proposition}
\newtheorem{corollary}{Corollary}
 
\newtheorem*{assump*}{Assumption}

\theoremstyle{remark}
\newtheorem{remark}{Remark}

\newcommand{\qednew}{\nobreak \ifvmode \relax \else
      \ifdim\lastskip<1.5em \hskip-\lastskip
      \hskip1.5em plus0em minus0.5em \fi \nobreak
      \vrule height0.75em width0.5em depth0.25em\fi}

\newcommand{\ump}{{p_{\min}}}
\title{Convergence Rate Analysis of the Multiplicative Gradient Method for PET-Type Problems}
\usepackage[colorlinks=true,breaklinks=true,bookmarks=true,urlcolor=blue,
     citecolor=blue,linkcolor=blue,bookmarksopen=false,draft=false]{hyperref}
\author{Renbo Zhao\thanks{MIT Operations Research Center, 77 Massachusetts Avenue, Cambridge, MA   02139 ({mailto:  renboz@mit.edu}).}}
\begin{document}

\maketitle

\begin{abstract}
We analyze the convergence rate of the multiplicative gradient (MG) method for PET-type problems with 
$m$ component functions and an $n$-dimensional optimization variable. 
We show that the MG method has an $O(\ln(n)/t)$ convergence rate, 
in both the ergodic and the non-ergodic senses. 
Furthermore, we show that the distances from the iterates to the set of optimal solutions converge (to zero) at rate $O(1/\sqrt{t})$. Our results show that, in the regime $n=O(\exp(m))$, to find an $\varepsilon$-optimal solution of the PET-type problems, the MG method has a lower computational complexity compared with the relatively-smooth gradient method and the Frank-Wolfe method for convex composite optimization involving a logarithmically-homogeneous barrier. %~\cite{Dvu_20,Zhao_20b}. %, as long as $n=O(\exp(m))$.  
\end{abstract}

%\begin{keyword}
%    Something else
%\end{keyword}

\section{Introduction}

We consider the following optimization problem:
\begin{align}
%\mathrm{(P)}: \quad 
f^*:={\max}_{x\in\Delta_n}\; \left\{f(x):=\textstyle\sum_{j=1}^m p_j\ln(a_j^\top x)\right\}, \tag{P} \label{eq:P}
\end{align}
where $\Delta_n :=\{x \in \bbR^n : x \ge 0, \sum_{i=1}^n x_i = 1\}$ denotes the (standard) unit simplex in $\bbR^n$,  $p_j> 0$ for all $j\in [m]$, and
%($e$ denotes the vector with all unit entries of proper dimension)
%where $p=(p_1,\ldots,p_m)$ is a probability vector, i.e., $p\in\Delta_m$, and 
$a_j\in\bbR_+^n$ for all $j\in [m]$. (Here $\bbR_+^n:= \{x\in\bbR^n:x_i\ge 0,\;\forall\,i\in[n]\}$, namely the nonnegative orthant in $\bbR^n$.)  %[0,+\infty)^n$.
We assume without loss of generality that $\sum_{j=1}^m p_j=1$. 
For well-posedness, we further assume that $a_j\ne 0$ for every $j\in[m]$, from which it follows that $\dom f\cap\Delta_n\ne \emptyset$ and hence \eqref{eq:P} has an optimal solution. %, which we denote by $x^*$. 
%(Note that this holds, in particular, if the data matrix $A = [a_1,\ldots,a_m]$ is ``sufficiently dense''.)

\subsection{Applications} \label{sec:app}

The problem \eqref{eq:P} subsumes many diverse applications, including most notably the positron emission tomography (PET) problem in medical imaging~\cite{Vardi_85}, computing the rate distortion in information theory~\cite{Csiszar_74},  maximum likelihood estimation for mixture models in statistics~\cite{Vardi_93}, the log-optimal investment problem~\cite{Cover_84}, and the maximum-likelihood inference of the multi-dimensional Hawkes processes~\cite{Zhou_13}. For motivation and clarity we now briefly describe the PET problem; for an expanded description we also refer the reader to~\cite[Section 1.1]{Zhao_20b} and \cite{BenTal_01}.   % due to its simplicity. 

PET is a medical imaging technique that measures the metabolic activities of human tissues and organs. In a typical setting radioactive materials are injected into the organ of interest, and these materials emit (radioactive) events %(e.g., gamma rays) 
that can be detected by PET scanners.  The mathematical model behind this process is described as follows. Suppose that an emission object (for example a human organ) has been discretized into $n$ voxels. The number of %(radioactive) 
events emitted by voxel $i$ ($i\in[n]$) is a Poisson random variable $\tilde X_i$ with {\em unknown} mean $x_i \ge 0$, and so $\tilde X_i\sim {\sf Poiss}(x_i)$, and  furthermore $\{\tilde X_i\}_{i=1}^n$ are assumed to be independent.  We also have a scanner array comprised of $m$ bins. Each event emitted by voxel $i$ has a {\em known} probability $p_{ij}$ of being detected by bin $j$ ($j\in[m]$), and we assume that $\sum_{j=1}^m p_{ij} = 1$, i.e., the event will be detected by exactly one bin.  
Let $\tilde Y_j$ denote the total number of events detected by bin $j$, whereby
\begin{equation}\label{thursday01}
\mathbb{E} [\tilde Y_j ]:= y_j := \textstyle\sum_{i=1}^n p_{ij} x_i \ . 
\end{equation}
By Poisson thinning and superposition, it follows that $\{\tilde Y_j\}_{j=1}^m$ are independent random variables 
and $\tilde Y_j\sim {\sf Poiss}(y_j)$ for all $j\in[m]$.  %Poission random variables such that $$

We seek to perform maximum-likelihood (ML) estimation of the unknown means $\{x_i\}_{i=1}^n$ based on observations $\{Y_j\}_{j=1}^m$of the random variables $\{\tilde Y_j\}_{j=1}^m$. From the model above, we easily see that the log-likelihood of observing $\{Y_j\}_{j=1}^m$ given $\{\tilde X_i\}_{i=1}^n$ is (up to some constants)
\begin{equation}
l(x):= - \textstyle\sum_{i=1}^n x_i + \textstyle\sum_{j=1}^mY_j\ln \big(\sum_{i=1}^n p_{ij} x_i\big) \ ,
\end{equation}
and therefore an ML estimate of $\{x_i\}_{i=1}^n$ is given by an optimal solution $x^*$ of %the problem 
\begin{equation}
{\max}_{x\ge 0} \;l(x) \ . \label{eq:PET_0}
\end{equation}
  %
%\begin{equation}
%\max \;\; L(x) \quad \st \;\;x \ge 0. 
%\end{equation}
% maximizer of $L$ over its domain $\bbR^n_+:=\{x\in\bbR^n:x\ge 0\}$.  
It follows from the first-order optimality conditions that %we see that 
any optimal solution $x$ must satisfy 
\begin{equation}
\textstyle\sum_{i=1}^n x_i = S := \sum_{j=1}^m Y_j \ .  \label{eq:PET_extra_constraint}
\end{equation}
By incorporating \eqref{eq:PET_extra_constraint} into \eqref{eq:PET_0}, and re-scaling both the objective function $l$ and the optimization variable $x$ by a factor of $S^{-1}$,  %$z:= x/S$, 
\eqref{eq:PET_0} can be equivalently written as 
\begin{equation}
{\max}_{z\in\Delta_n} \; \textstyle\sum_{j=1}^m p_j\ln \big(\sum_{i=1}^n p_{ij} z_i\big) \ , %\quad \st\;\; {z\in\Delta_n} \ . 
\label{eq:PET_final}
\end{equation} 
where $p_j:= Y_j/S$ for all $j\in[m]$.  The maximization problem \eqref{eq:PET_final} is easily seen to be an instance of \eqref{eq:P} with $a_j := (p_{j1}, \ldots, p_{jn})^\top$ for $j \in [m]$. 

\subsection{The Multiplicative Gradient (MG) Method}

% This problem has numerous applications,  
One of the earliest algorithms for solving~\eqref{eq:P} is the MG method, which seems to be first proposed by information theorists in the 1970s for computing channel capacity and rate-distortion functions~\cite{Arimoto_72,Blahut_72,Csiszar_74}. The MG method is very simple and can be described as follows. Let $x^0\in\ri\Delta_n$ be the initial point, where $\ri\Delta_n:= \{x \in \bbR^n : x > 0, \sum_{i=1}^n x_i = 1\}$ denotes the relative interior of $\Delta_n$.  Then at each iteration $t \ge 0$ of the MG method we first compute the gradient of $f$ at $x^t$:
$$ \nabla_i f(x^t)= \sum_{j=1}^m p_j\frac{a_{ji}}{a_j^\top x^t} \  , \ \forall\,i \in [n] \ , $$ 
where $\nabla_i f(x^t)$ denotes the $i$-th entry of $\nabla f(x^t)$, and $a_{ji}$ denotes the $i$-th entry of $a_j$.  We construct the next iterate $x^{t+1}$ by simply multiplying the current iterate's coefficients by their respective gradient coefficients, namely
\begin{equation}
x^{t+1}_i := x^t_i \nabla_i f(x^t)    \ , \ \forall\, i\in[n] \ . \label{eq:MG}
\end{equation}
 %To understand this method, 
Let us make two immediate (but important) observations about the MG method in~\eqref{eq:MG}. Define the $m\times n$ (nonnegative) data matrix 
\begin{equation}
A:= [a_1 \;\;\cdots\;\; a_m]^\top = [A_1 \;\;\cdots\;\; A_n], \label{eq:A}
\end{equation}
(i.e., $a_j^\top$ denotes the $j$-th row of $A$ for $j\in[m]$ and $A_i$ denotes the $i$-th column of $A$ for $i\in[n]$), and let $\calI:= \{i\in[n]:A_i\ne 0\}$ denote the ``support pattern'' of the columns of $A$. 
Then we know that %$\{x^t\}_{t\ge 1}\subseteq $
\begin{enumerate}[label = (O\arabic*)]
\item \label{item:feas}  the sequence $\{x^t\}_{t\ge 0}\subseteq \Delta_n$ (namely, $\{x^t\}_{t\ge 0}$ are feasible) and % the iterates generated by the MG method remain feasible,
\item \label{item:supp} for all $t\ge 1$, $\calI(x^t) = \calI$, where %given any $x\in\Delta_n$, 
$\calI(x):= \{i\in[n]:x_i\ne 0\}$ denotes the {\em support} of $x$ (for any $x\in\bbR^n$). 
\end{enumerate}
To see~\ref{item:feas}, note that $x^0\in \Delta_n$, and if $x^t\in \Delta_n$ for some $t\ge 0$, then $x^{t+1}\ge 0$ since $\nabla f(x^t)\ge 0$ and 
\begin{equation}
\sum_{i=1}^n x^{t+1}_i = \sum_{i=1}^n x^t_i \sum_{j=1}^m p_j\frac{a_{ji}}{a_j^\top x^t} = \sum_{j=1}^m p_j=1 \ .
\end{equation}
This shows that $x^{t+1}\in \Delta_n$. To see~\ref{item:supp}, note that $\calI(x^1) = \calI$ since $x^0>0$ and $\calI(\nabla f(x^0)) = \calI$, and if  $\calI(x^t) = \calI$ for some $t\ge 1$, then $a_j^\top x^t>0$ for all $j\in[m]$ and hence  $\calI(\nabla f(x^t)) = \calI$.  As a result, we have $\calI(x^{t+1}) = \calI$. 

%As a sanity check, note that if $x^t\in \Delta_n$, then we always have $x^{t+1}\in \Delta_n$ since $\nabla_i f(x^t)\ge 0$ and hence $x^{t+1}_i\ge 0$ for all $i\in[n]$, and 
%
%In addition, note that if the MG method terminates in a finite number of iterations, meaning that $x^{t+1}=x^t$ for some $t\ge 0$, then $\nabla f(x^t)=e$ where $e:=(1, \ldots, 1)^\top$ denotes the vector of ones. Since the constraint set is the unit simplex $\Delta_n$, it then follows that $\lranglet{\nabla f(x^t)}{x-x^t}=0$ for all $x\in\Delta_n$.  Therefore  $x^t$ satisfies the first-order optimality conditions for \eqref{eq:P}, and so by convexity it follows that $x^t$ is an optimal solution of \eqref{eq:P}. %This shows that the multiplicative update in~\eqref{eq:EM} is quite specific to the simplex constraint. 

\subsection{Motivation}\label{sec:mot}

As mentioned in Section~\ref{sec:app}, although the problem~\eqref{eq:P} appears to be structurally simple, it includes many important applications across a variety of fields. Unfortunately, due to the presence of $\ln(\cdot)$, the objective function $f$ is neither Lipschitz nor has Lipschitz gradient on the constraint set $\Delta_n$, and this prohibits us from applying most of the traditional first-order methods~\cite{Nemi_79,Nest_04} to solve~\eqref{eq:P}. 

In the literature, the standard choice for solving~\eqref{eq:P} is the MG method, and it has been long known that the sequence of iterates generated by the MG method has a {\em unique} limit point that is an optimal solution of~\eqref{eq:P}  (see e.g.,~\cite{Csiszar_84,Vardi_85,Iusem_92}). However, the convergence rate of this method has remained unclear for almost fifty years. Recently, some ``unconventional'' first-order methods have been proposed to minimize certain differentiable functions whose gradients are not Lipschitz on the constraint set. Among them, two methods are applicable to~\eqref{eq:P}, namely, the relatively-smooth gradient method (RSGM)~\cite{Bauschke_17,Lu_18} and the Frank-Wolfe method for convex composite optimization involving a logarithmically-homogeneous barrier (FW-LHB)~\cite{Dvu_20,Zhao_20b}.

Given the methods mentioned above, a natural question is to determine which method is ``best suited'' for solving~\eqref{eq:P}. Numerically, the extensive experimental results in~\cite[Section~4.2]{Zhao_20b} indicate that  the MG method {\em significantly and consistently outperforms} RSGM and  FW-LHB across different values of $m$ and $n$, and regardless of the location of the initial point $x^0$ (namely, whether $x^0$ lies at the center of $\Delta_n$ or lies close to the relative boundary of $\Delta_n$). The extraordinary numerical performance of the MG method is rather surprising and somewhat mysterious, for two reasons. First, the structure of the MG method is extremely simple. Indeed, compared with the other two methods, the MG method does not require selecting step-sizes, solving a projection sub-problem onto the constraint set $\Delta_n$, or solving a linear minimization sub-problem over $\Delta_n$. %, or solving a linear minimization subproblem on the feasible region. 
Second, unlike the other two methods with known convergence rates, the  MG method has only been known to converge asymptotically (i.e., without any convergence rate guarantees). % (under certain conditions), 

The discussion above motivates us to investigate the computational guarantees of the MG method, with the hope that the results could provide theoretical justification for the superior numerical performance of the MG method. 

%Given that ,  % and somewhat mysterious, \\Indeed, this leads us to 

%The MG method is a very attractive method for solving \eqref{eq:P} for at least two reasons. The first reason is simplicity: unlike other other first-order methods, the MG method does not require selecting step-sizes, solving a projection problem onto the feasible region, or solving a linear minimization subproblem on the feasible region. The second reason is good empirical performance. Indeed, as noted in some recent references (e.g.,~\cite[Section~4]{Dvu_20,Zhao_20b}), the MG method empirically runs faster compared to several standard first-order methods, including the Frank-Wolfe method~\cite{Dvu_20,Dvu_20,Zhao_20b}, the Bregman proximal gradient method~\cite{Bauschke_17,Lu_18} and the primal-dual hybrid gradient method~\cite{Chambolle_18}. Despite these attractive features of the MG method, to the best of our knowledge there have been no convergence rate guarantees for the algorithm, in contrast to its known asymptotic convergence properties (see e.g.,~\cite{Csiszar_84,Cover_84,Vardi_85,Iusem_92}). 

\subsection{Contributions}\label{sec:contributions}

Our contributions are summarized as follows. 

\begin{enumerate}[label=\roman*)]
\item We present the first (to our knowledge) convergence rate analysis for the MG method.  We show -- via a surprisingly simple proof -- that the MG method has an $O(\ln(n)/t)$ convergence rate in both the ergodic and the non-ergodic cases. (Recall that $n$ denotes the dimension of the optimization variable $x$.)% As a corollary of our results, 
\item We provide an extremely short proof of the asymptotic convergence of $\{x^t\}_{t\ge 0}$ to an optimal solution of~\eqref{eq:P}, which is much simpler than the existing proofs in~\cite{Csiszar_84,Vardi_85,Iusem_92}. 
\item We show that the distance from $x^t$ to the set of optimal solutions of~\eqref{eq:P} converges to zero at rate $O(1/\sqrt{t})$, by constructing a quadratic error bound of~\eqref{eq:P}. %using the self-concordance of the log-barrier function and Hoffman's lemma~\cite{Hoffman_52}, 
\item We derive the computational complexities of RSGM and FW-LHB for finding an $\varepsilon$-optimal solution of~\eqref{eq:P} with initial point $x^0 = (1/n)e$, where $e:=(1,\ldots,1)\in\bbR^n$. 
%By comparing the computational complexity of the MG method with that of RSGM and FW-LHB, we 
This in turn enables us to conclude that  % two other methods (namely, RSGM and FW-LHB), in terms of . Our results show that, 
the computational complexity of the MG method is always lower than that of RSGM, and is also lower than that of FW-LHB in the regime $n=O(\exp(m))$.  % to the two other methods. 
\end{enumerate}

 %As a by-product of our analysis, we also show -- again by a very simple argument -- that the sequence of iterates $\{x^t\}_{t\ge 0}$ has a {\em unique} limit point that is an optimal solution of~\eqref{eq:P}, and therefore greatly simplifies the previous uniqueness proofs of~\cite{Cover_84} and \cite{Iusem_92}).  
%Furthermore,   % -- measured in any norm on $\bbR^n$ -- 

 %Our proof is surprisingly simple --- indeed, it is simpler (at least no more complicated) than the aforementioned asymptotic analyses. In addition, our proof enables us to develop a generalization of the MG method that enjoys accelerated convergence rates, and we show that the accelerated rate depends on how fast the ``step-sizes'' can grow. Indeed, the original MG method can be regarded as a special instance of this generalized algorithm with unit step-sizes. Lastly, we interpret the generalized MG method as the mirror descent method applied to \eqref{eq:P} with ``log-gradient''. 
% version, and  can be regarded as a special   which subsumes  as case.  
 % in the literature. %guarantees in e.g.,~\cite{Cover_84,Csiszar_84,Vardi_85}. 

%of the algorithm~\eqref{eq:EM} is based on Csiszar's alternating minimization framework. However, no  In this note, we will give a $O(1/t)$ convergence rate. Our analysis is surprisingly simple 

\section{Convergence Rate Analysis} \label{sec:conv_analysis}

In this section we present the convergence analysis of the MG method, where we will prove an $O(\ln(n)/t)$ convergence rate for both the non-ergodic and the ergodic cases. %, where $\bar D$ is a certain Bregman divergence which will shortly be formally defined. 

Before doing so, let us introduce some definitions and conventions. 
Let $\calX^*\ne \emptyset$ be the set of optimal solutions of \eqref{eq:P}, and let $x^*$ be any point in $\calX^*$, so that $f^*= f(x^*)$.  For any $x,y\ge 0$, define the Kullback-Leibler (KL) divergence 
\begin{equation}
D_{\rm KL}(y,x):=  \textstyle\sum_{i=1}^n y_i\ln(y_i/x_i) \ ,%\quad \forall\,y\ge 0,\; \forall\,x>0 \ , 
\label{eq:def_KL}
\end{equation}
where we use the following conventions:
\begin{equation}
0\ln 0 := 0, \quad 0\ln (0/0) := 0 \quad\mbox{and} \quad a\ln (a/0) := +\infty,\quad \forall\,a>0.  \label{eq:convention} 
\end{equation}
Using these conventions and~\cite[Lemma~3]{Nest_05}, we know that  
\begin{equation}
D_{\rm KL}(y,x)\ge (1/2)\normt{y-x}_1^2\ge 0, \quad \forall\, y,x \in\Delta_n, \label{eq:lb_KL}
\end{equation}
where $\normt{\cdot}_1$ denotes the $\ell_1$-norm, namely, $\normt{x}_1:= \sum_{i=1}^n \abst{x_i}$ for  $x\in\bbR^n$. 
(To see why~\eqref{eq:lb_KL} holds, note that if the support $\calI(y)\not\subseteq\calI(x)$, then $D_{\rm KL}(y,x)=+\infty$ and~\eqref{eq:lb_KL} trivially holds; otherwise, we can restrict~\eqref{eq:lb_KL} to the ``sub-simplex'' $\Delta_{\abst{\calI(x)}}$ and apply~\cite[Lemma~3]{Nest_05}.)

%Let the entropy function be given by $h(x):= \sum_{i=1}^n x_i\ln x_i - x_i$ for $x\ge 0$ (where $a \ln(a) := 0$ for $a=0$).  Then $h$ is a convex function defined on $\bbR^n_+$, and the Bregman divergence induced by $h$ is given by
%\begin{equation}
%D_{\rm KL}(y,x):= h(y)- h(x) - \lranglet{\nabla h(x)}{y-x} = \textstyle\sum_{i=1}^n y_i\ln(y_i/x_i) \ ,\quad \forall\,y\ge 0,\; \forall \,x>0 \ , \label{eq:def_KL}
%\end{equation}
%which is known as the Kullback-Leibler (KL) divergence. From, we know that %it is well-known that 

%We note that the convexity of $h$ implies that $D_{\rm KL}(y,x) \ge 0$ for all $y \ge 0$ and $x > 0$. 
 Our main convergence results are stated as follows:

\begin{theorem} \label{thm:O(1/t)}
Let $\{x^t\}_{t\ge 0}$ be the iterates of the MG method. Then for any $x^*\in\calX^*$ and all $t\ge 0$: %we have:
\begin{enumerate}[label = (\roman*)]
\item Non-ergodic rate: $f^* - f(x^t) \le \frac{D_{\rm KL}(x^*,x^0)}{t+1}$, and   \label{item:non_ergodic}
\item Ergodic rate: $f^* - f(\barx^t) \le \frac{D_{\rm KL}(x^*,x^0)}{t+1}$, where $\barx^t:=\tfrac{1}{t+1}\sum_{k=0}^t x^k$. \qed \label{item:ergodic}

\end{enumerate}

\end{theorem} 

Before proving Theorem~\ref{thm:O(1/t)}, let us first state an important corollary. % of it. 

\begin{corollary}\label{cor:data_indep_rate}
%For any $x\in\Delta_n$ and $x^0\in\ri\Delta_n$, we have $$D_{\rm KL}(x,x^0)\le -\ln(x^0_{\min}) \ ,\quad \mbox{ where }\;\; x^0_{\min}:= {\min}_{i\in[n]}\; x^0_i \ .$$ 
Let $\{x^t\}_{t\ge 0}$ be the iterates of the MG method. The for all $t\ge 0$, we have  %For any $x^0\in\ri\Delta_n$, we have
\begin{equation}
f^* - f(x^t) \le \frac{-\ln(x^0_{\min})}{t+1} \quad \mbox{and}\quad f^* - f(\barx^t) \le \frac{-\ln(x^0_{\min})}{t+1} \ . \label{eq:rate_data_indep}
\end{equation}
Consequently, if we choose $x^0 = (1/n)e$, %(where $e:=(1, \ldots, 1)^\top$ denotes the vector of ones), 
then for all $t\ge 0$, we have 
\begin{equation}
f^* - f(x^t) \le \frac{\ln(n)}{t+1} \quad \mbox{and}\quad f^* - f(\barx^t) \le \frac{\ln(n)}{t+1} \ . \label{eq:rate_data_indep_ln(n)}
\end{equation}
%the MG method converges at rate $O(\ln(n)/t)$, in both the non-ergodic and the ergodic cases.
\end{corollary}

\begin{proof}
Since $D_{\rm KL}(\cdot,x^0)$ is convex for any $x^0\in\ri\Delta_n$, we have
\begin{equation}
{\max}_{x\in\Delta_n}\; D_{\rm KL}(x,x^0) = {\max}_{i\in[n]}\; D_{\rm KL}(e_i,x^0) = {\max}_{i\in[n]}\; -\ln(x^0_i) = -\ln(x^0_{\min})\ , \nn 
\end{equation}
where $e_i$ is the $i$-th standard coordinate vector in $\bbR^n$ for $i\in[n]$. Then based on Theorem~\ref{thm:O(1/t)}, we arrive at~\eqref{eq:rate_data_indep}. This completes the proof.  %Since $x^0\in\ri\Delta_n$, we have $x^0_{\min}\le 1/n$, 
\end{proof}

\begin{remark} 
Note that in Theorem~\ref{thm:O(1/t)}, we provide data-dependent convergence rates --- this is because $D_{\rm KL}(x^*,x^0)$ depends on the optimal solution $x^*\in\calX^*$, which in turn depends on the data matrix $A$ (cf.~\eqref{eq:A}). In contrast, in Corollary~\ref{cor:data_indep_rate} we provide data-independent convergence rates. % (i.e., they do not depend on  $\calX^*$).
  In fact, these rates only depend on the minimum element of the starting point $x^0$. From~\eqref{eq:rate_data_indep}, it is clear that the optimal choice of $x^0$ should be $(1/n)e$, which yields $O(\ln(n)/t)$ convergence rates in both the non-ergodic and the ergodic cases. 
 %$n$, the dimension of the optimization variable. 
%Note that in either~\ref{item:non_ergodic} or~\ref{item:ergodic}, we provided two rates, one being data-dependent, i.e.,  ${D_{\rm KL}(x^*,x^0)}/({t+1})$ and the other being data-independent, i.e., ${\ln(n)}/({t+1})$. The latter rate is simply obtained by  Depending on the location  
%In the typical case when $x^0 = \tfrac{1}{n}e$ it is easy to show $D_{\rm KL}(x^*, x^0) \le \ln(n)$, and therefore in this case the numerators in the above bounds are themselves simply bounded above by $\ln(n)$. 
\end{remark}

Towards the task of proving Theorem \ref{thm:O(1/t)}, we begin our analysis with the following simple but important observations about \eqref{eq:P}. 

\begin{lemma}\label{lem:KKT}
There exists $\nu\in\bbR_+^n$ such that $\nabla_i f(x^*) + \nu_i = 1$ and $\nu_ix^*_i=0$, for all $i\in[n]$. In particular, if $x^*_i>0$ for some $i\in[n]$, then $\nabla_i f(x^*)=1$.  
\end{lemma}
\begin{proof}
First, let us observe that for any $x\in\Delta_n$ we have
\begin{equation}
\lranglet{\nabla f(x)}{x} =  \sum_{i=1}^n x_i \sum_{j=1}^m p_j\frac{a_{ji}}{a_j^\top x} = \sum_{j=1}^m p_j = 1 \ , \label{eq:log_hom}
\end{equation}
where $a_{ji}$ denotes the $i$-th entry of $a_j$.  
%(Note that~\eqref{eq:log_hom} can also been seen from that $-f$ is $1$-logarithmically homogeneous; see e.g.,~\cite[Section 2.3.5]{Rene_01}.) 
The KKT conditions are necessary and sufficient for optimality for \eqref{eq:P}, and thus there exists $\lambda\in\bbR$ and $\nu\in\bbR_+^n$ such that $\nabla f(x^*) + \lambda e + \nu = 0$, and $\nu_ix^*_i=0$ for all $i\in[n]$ (where recall that $e$ denotes the vector of ones). Consequently
\begin{align*}
1+\lambda=\lranglet{\nabla f(x^*)}{x^*} + \lambda \lranglet{e}{x^*} + \lranglet{\nu}{x^*} = 0 \ ,
\end{align*}
which implies that $ \lambda = -1 $ and hence $\nabla f(x^*)  + \nu = e$. 
\end{proof}

The next lemma, which is due to Cover~{\cite[Theorem~1]{Cover_84}}, presents a lower bound on the improvement of the objective value at each iteration of the MG method. For completeness, we include a short proof in Appendix~\ref{app:proof}.

\begin{lemma}[Cover~{\cite[Theorem~1]{Cover_84}}]\label{lem:monotonicity}
For all $t\ge 0$ we have 
\begin{equation}
f(x^{t+1}) - f(x^t)\ge D_{\rm KL}(x^{t+1},x^{t})\ge 0 \ . \tag*{\qed} 
\end{equation}
\end{lemma} %\qed 

%\begin{proof}
%
%\end{proof}

%Note that 
For convenience, let us define the objective gap at $x\in\Delta_n$ as $\delta(x) := f^* - f(x)$; then % and $\delta_t:= \delta(x^t)$ for $t\ge 0$. 
from Lemma~\ref{lem:monotonicity} we see that $\{\delta(x^t)\}_{t\ge 0}$ is a monotonically 
non-increasing sequence.

\begin{lemma} \label{lem:ub_gap}
For any $x\in\Delta_n$ we have 
\begin{equation}
\delta(x)\le \textstyle\sum_{i=1}^n x_i^*\ln\left(\nabla_i f(x)\right). % = \lrangle{\ln(\nabla f(x))}{x^*}, 
\end{equation}
%where we overload the notation $\ln(\cdot)$ such that it applies entry-wise to a vector.  %$\ln(x) := (\ln(x_i))_{i=1}^n$ for any  $x>0$. 
\begin{proof}
Define $\calI^*:= \{i\in[n]: x_i^*>0\}$, and from Lemma~\ref{lem:KKT} it follows that for any $i \in \calI^*$ we have
\begin{equation}
\sum_{j=1}^m \frac{p_ja_{ji}}{a_j^\top x^*} =\nabla_i f(x^*) = 1 \ . \label{eq:cvx_coeff}
\end{equation}
As a result we have
\begin{align*}
\sum_{i=1}^n x_i^*\ln\left(\nabla_i f(x)\right)&= \sum_{i\in\calI^*} x_i^*\ln\left(\nabla_i f(x)\right)\\
&= \sum_{i\in\calI^*} x_i^*\ln\left(\sum_{j=1}^m \frac{p_ja_{ji}}{a_j^\top x^*}\frac{a_j^\top x^*}{a_j^\top x}\right)\\
&\ge \sum_{i\in\calI^*} x_i^* \sum_{j=1}^m \frac{p_ja_{ji}}{a_j^\top x^*}\ln\left(\frac{a_j^\top x^*}{a_j^\top x}\right)\\
&= \sum_{j=1}^m p_j\ln\left(\frac{a_j^\top x^*}{a_j^\top x}\right)\\
&= f(x^*) - f(x) = \delta(x) \ ,  \nt \label{eq:lb_log_grad}
\end{align*}
where the inequality above uses~\eqref{eq:cvx_coeff} and the concavity of $\ln(\cdot)$.
\end{proof}
\end{lemma}

% We give the $O(1/t)$ convergence rate of~\eqref{eq:EM} in the following theorem.

Equipped with the above lemmas, we now prove Theorem \ref{thm:O(1/t)}.  

\begin{proof}[Proof of Theorem~\ref{thm:O(1/t)}]
%We use $D_{\rm KL}(x^*,x^t)$ as the Lyapunov function. % and define $\calI$ as in the proof of Lemma~\ref{lem:ub_gap}. %:= \{i\in[n]: x_i^*>0\}$. 
%For any $i\in\calI$, in Lemma~\ref{lem:KKT}, we have $\nu_i = 0$ and hence 
For any $x^*\in\calX^*$, we have 
\begin{align}
D_{\rm KL}(x^*,x^t) - D_{\rm KL}(x^*,x^{t+1}) &= \sum_{i=1}^n x_i^*\ln\left(\frac{x^{t+1}_i}{x^t_i}\right) = \sum_{i=1}^n x_i^*\ln\left(\nabla_i f(x^t)\right)\ge \delta(x^t) \ , \label{eq:dec_KL}
\end{align}
where the inequality on the right follows from Lemma~\ref{lem:ub_gap}. 
Telescoping over $k= 0,\ldots,t$, we obtain
\begin{align*}
D_{\rm KL}(x^*,x^0) \ge D_{\rm KL}(x^*,x^0) - D_{\rm KL}(x^*,x^{\textcolor{black}{t+1}})\ge \textstyle \sum_{k=0}^{t} \delta(x^k) \ . %%\ge t\delta_{t-1}. 
\end{align*}
Using the convexity of $\delta(\cdot)$, we have $\sum_{k=0}^{t} \delta(x^k)\ge (t+1)\delta(\barx^t)$, which proves part~\ref{item:ergodic}. Alternatively, using the property that $\{\delta(x^t)\}_{t\ge 0}$ is a non-increasing sequence, we have $\sum_{k=0}^{t} \delta(x^k)\ge (t+1)\delta(x^t)$, which then proves part~\ref{item:non_ergodic}. % of the theorem.
%This completes the proof. 
\end{proof}

\begin{remark} \label{rmk:Iusem}
Under the ``normalization'' assumption that $\sum_{j=1}^m a_{ji} = 1$ for all $i\in[n]$, Iusem~\cite[Lemma 2.2]{Iusem_92} showed a recursion that is slightly more general than~\eqref{eq:dec_KL}.
%Note that if this assumption holds, then the data matrix $A$ (cf.~\eqref{eq:A}) cannot have zeros columns, which is unnecessarily restrictive. In addition, normalizing the (non-zero) columns of $A$ requires re-scaling the optimization variable $x$, which results in a non-simplex constraint set that prohibits the application of the MG method.  
 %For example,  the data matrix PET application as introduced in Section~\ref{sec:app} where the 
In fact, the proof technique in~\cite{Iusem_92} mainly leverages the joint convexity of the KL-divergence, which is quite different from our proof above. In addition, it is not clear (to us) if the technique in \cite{Iusem_92} can still be applied in the general setting where the normalization assumption is absent. 
\end{remark}

Next, we present an extremely short proof of the asymptotic convergence of $\{x^t\}_{t\ge 0}$ to some $x^*\in\calX^*$, which is much simpler than the existing proofs in~\cite{Csiszar_84,Vardi_85,Iusem_92}.

\begin{corollary} \label{cor:asymp}
There exists some $x^*\in\calX^*$ such that  %$D_{\rm KL}(x^*,x^t)\to 0$ and hence 
$x^t\to x^*$. % as $t\to +\infty$. 
\end{corollary}

\begin{proof}
%Let $\barx$ be any limit point of $\{x^t\}_{t\ge 0}$ (which exists since $\Delta_n$ is compact), such that it is the limit of some sub-sequence $\{x^{t_l}\}_{l\ge 0}$. 
Since $\Delta_n$ is compact, there exists a sub-sequence of $\{x^t\}_{t\ge 0}$, which we denote by $\{x^{t_l}\}_{l\ge 0}$, that converges to some $\barx\in\Delta_n$. 
Using the conventions in~\eqref{eq:convention}, it is clear that $D_{\rm KL}(\barx,x^{t_l})\to 0$ (as $l \to +\infty$). 
%We know that $D_{\rm KL}(x^*,x^{t_l}) = \sum_{i\in\calI}x^*_i\ln(x^*_i/x^{t_l}_i)\to 0$, where $\calI:= \{i\in[n]: x_i^*>0\}$. 
Since $f(x^t)\to f^*$ (cf.\ Theorem~\ref{thm:O(1/t)}\ref{item:non_ergodic}), we see that $f(x^{t_l})\to f^*$ and hence $\barx\in\calX^*$. 
On the other hand, from~\eqref{eq:dec_KL}, we know that the nonnegative sequence $\{D_{\rm KL}(\barx,x^t)\}_{t\ge 0}$ is non-increasing, and hence $D_{\rm KL}(\barx,x^t)\to d^*$ for some $d^*\ge 0$. Since $D_{\rm KL}(\barx,x^{t_l})\to 0$, we know that $d^*=0$, which implies that $D_{\rm KL}(\barx,x^t)\to 0$. Finally, since $\barx\in\Delta_n$ and $x^t\in\Delta_n$ for all $t\ge 0$, we use~\eqref{eq:lb_KL} to conclude that $\normt{x^t-x^*}_1\to 0$. This  completes the proof.
%From Theorem~\ref{thm:O(1/t)}\ref{item:non_ergodic}, we know that $f(x^t)\to f^*$, and so any limit point $$
\end{proof}

%Based on Theorem~\ref{thm:O(1/t)} and Lemma~\ref{lem:monotonicity}, we easily have the following corollary. % regarding the convergence of the sequence of iterates $\{x^t\}_{t\ge 0}$. 
%
%\begin{corollary}\label{cor:asym_iter_conv}
%For all $t\ge 0$, we have %$\norm$
%\begin{equation}
%\normt{x^{t+1} - x^t}_1\le \frac{\sqrt{2D_{\rm KL}(x^*,x^0)}}{\sqrt{t+1}} \ .  \label{eq:rate_succ_iter}
%\end{equation}
%%Consequently, %we have $x^t\to \barx \in\calX^*$, namely 
%%the sequence of iterates $\{x^t\}_{t\ge 0}$ has a unique limit point $ \barx \in\calX^*$. 
%\end{corollary}
%
%\begin{proof}
%We have 
%\begin{equation}\label{cyt}
%\delta(x^t) = f(x^*) - f(x^t)\ge f(x^{t+1}) - f(x^t)\gea D_{\rm KL}(x^{t+1},x^t)\geb \tfrac{1}{2}\normt{x^{t+1}-x^t}_1^2 \ , 
%\end{equation}
%where (a) follows from Lemma~\ref{lem:monotonicity}, and (b) uses the $1$-strong-convexity of $h$ on $\Delta_n$ with respect to the norm $\normt{\cdot}_1$ (see~\cite{Nest_05} for a short proof thereof). Combining \eqref{cyt} with Theorem~\ref{thm:O(1/t)}\ref{item:non_ergodic} yields~\eqref{eq:rate_succ_iter}. 
%%Since $\normt{x^{t+1}-x^t}_1\to 0$, and $\Delta_n$ is compact, we see that $\{x^t\}_{t\ge 0}\subseteq \Delta_n$ has a unique limit point $\barx\in\Delta_n$, and the continuity of $f$ implies that $f(x^t)\to f(\barx)$. Again, from Theorem~\ref{thm:O(1/t)}\ref{item:non_ergodic}, we see that $f(\barx)= f^*$ and hence $\barx \in\calX^*$. 
%\end{proof}

\section{Error Bound and Convergence Rate of the Distance to $\calX^*$}\label{sec:SC_EB} %$\dist(x^t,)$}

%In Corollary~\ref{cor:asym_iter_conv}, we have shown that $\{x^t\}_{t\ge 0}$ converges to some optimal solution  $x^* \in\calX^*$ of \eqref{eq:P}. %In particular, we have $x^t\to \calX^*$ as $t\to+\infty$. 
In this section we present an error bound for $f$ on the constraint set $\Delta_n$, %(i.e., the feasible region of~\eqref{eq:P}), %on the distance of any $x\in\Delta_n$ (i.e., the feasible region of~\eqref{eq:P}) to the set of optimal solutions $\calX^*$, 
which we then use to characterize the rate of convergence of the distances from $\{x_t\}_{t\ge 0}$ to $\calX^*$. 
%of the iterates of the MG method to the set of optimal solutions. 
 Let $\normt{\cdot}$ be any given norm on $\bbR^n$, and define 
\begin{equation}
\dist_{\normt{\cdot}}(x,\calX^*):= {\min}_{x^*\in\calX^*}\; \normt{x-x^*} \ , \quad \forall\, x\in\bbR^n. \label{eq:def_dist}
\end{equation}
%for any $x\in\bbR^n$.  %namely the $\ell_1$-distance from $x$ to $\calX^*$. 
%We remark that due to the norm equivalence in finite-dimensional (normed) spaces, the choice of  $\normt{\cdot}_1$ in~\eqref{eq:def_dist} is not essential, and replacing it with another norm only affects the convergence rate by a constant independent of $t$. 
%from $x^t$

%The crux of our analysis lies in establishing an error bound of $f$ on $\Delta_n$. Before deriving this bound, 
Let us introduce some new notations. Let %$A:= [a_1,\ldots,a_m]^\top\in\bbR^{m\times n}$ and %be the linear operator formed by concatenating the linear functions $a_j^\top x$ for $j\in[m]$, and let 
$\calY:= A(\Delta_n)\subseteq \bbR^m$, which is the image of $\Delta_n$ under the linear operator $A$ (cf.~\eqref{eq:A}). %Define $\barf(y):= \sum_{j=1}^m p_j \ln y_j$ for all $y>0$ and 
Note that we can write \eqref{eq:P} equivalently as
\begin{align*}
\quad {\max}_{y\in\calY}\; [\barf(y):= \textstyle\sum_{j=1}^m p_j \ln y_j]   \tag{$\mathrm{P}_y$} \label{eq:P_y} \ ,
\end{align*}
where $\dom \barf = \bbR_{++}^m:= \{y\in\bbR^m:y>0\}$. 
The strict concavity of $\barf$ implies that \eqref{eq:P_y} has a unique optimal solution $y^*\in\calY\cap\dom \barf = \calY\cap \bbR_{++}^m$, % (where $\bbR_{++}^m:=\{y\in\bbR^m:y_j>0,\,\forall\,j\in[m]\}$), 
and  hence we have $f^*= \barf(y^*)$ and 
\begin{equation}\label{runtime}
\calX^*:= \{x\in\bbR^n_+:e^\top x = 1, \; Ax=y^*\} \ .
\end{equation}
In particular, we can write 
\begin{equation}
\calX^*= \Delta_n\cap\calL, \quad\mbox{where}\quad \calL:= \{x\in\bbR^n: Ax = y^*\}.  \label{eq:char_X*}
\end{equation}
%By defining the affine subspace $\calL:= \{x\in\bbR^n: Ax = y^*\}$, we can $\calX^*= \Delta_n\cap\calL$. 
%In the following, we assume that $\abst{\calY}>1$ (otherwise \eqref{eq:P_y} and hence \eqref{eq:P} is trivial). 
Let us define the following norm on $\bbR^m$ induced by the Hessian of $\bar f$ at $y^*$:
\begin{equation}\label{norman}
\normt{y}_{y^*}: = \sqrt{\lranglet{-\nabla^2 \barf(y^*) y}{y}} = \sqrt{\textstyle\sum_{j=1}^m p_j(y_j/y^*_j)^2} \ ,\quad \forall\,y\in \bbR^m \ ,
\end{equation} and define the ``radius'' of $\calY$ centered at $y^*$ as
\begin{equation}
R_{y^*} := {\max}_{y\in\calY}\, \normt{y-y^*}_{y^*}= {\max}_{x\in\Delta_n}\, \normt{Ax-y^*}_{y^*} = {\max}_{i \in [n]}\, \normt{A_i - y^*}_{y^*} \ . \label{eq:def_R_y^*}
\end{equation}
(Recall that $A_i$ is the $i$-th column of $A$ for $i \in [n]$.)
%Next, we aim to establish an error bound of $f$ on $\Delta_n$. The construction procedure follows from the three lemmas below. 
In addition, let $\ump:= \min_{j\in[m]}\, p_j$, and note that since $\sum_{j=1}^m p_j=1$, we have $\ump \in(0,1/m]$.  The next lemma uses the above notations and definitions to construct a lower bound of the optimality gap $f^*-\barf$ on $\calY$.

\begin{lemma} \label{lem:EB_y}
%Define  and the (local) norm $\normt{y}_{y^*}:= \lranglet{-\nabla^2 \barf(y^*)y}{y}^{1/2}$ for all $y\in\bbR^m$. 
For all $y\in\calY$ we have 
\begin{equation}
 f^* - \barf (y)\ge  \ump\, \omega\big(p_{\min}^{-1/2}\normt{y-y^*}_{y^*}\big) \ ,
\end{equation}where $\omega(t):= t - \ln(1+t)$ for $t > -1$.
\end{lemma}

\begin{proof}
Define $F := p_{\min}^{-1}\barf$, and from standard results on self-concordant function theory (e.g.,~\cite[Theorem 2.2.6]{Rene_01}), we observe that $-F$ is a (standard strongly non-degenerate) self-concordant function with $\dom F = \dom \barf = \bbR^m_{++}$. Therefore from~\cite[Theorem 4.1.7]{Nest_04}, we have for all $y\in\calY$:
\begin{equation}
-F(y)\ge -F(y^*) - \ipt{\nabla F(y^*)}{y-y^*} + \omega\big(\sqrt{\lranglet{-\nabla^2 F(y^*)(y-y^*)}{y-y^*}}\big) \ , 
\end{equation}
which is equivalent to
\begin{align*}
\barf(y)&\le \barf(y^*) + \ipt{\nabla \barf(y^*)}{y-y^*} - \ump \, \omega\big(p_{\min}^{-1/2}\normt{y-y^*}_{y^*}\big)\\
 &\le \barf(y^*)  - \ump \, \omega\big(p_{\min}^{-1/2}\normt{y-y^*}_{y^*}\big) \ , 
\end{align*}
where the last inequality uses $\ipt{\nabla \barf(y^*)}{y-y^*}\le 0$ for all $y\in\calY$. 
\end{proof}

\begin{lemma}%[{\cite{Zhao_21}}] 
\label{lem:omega}
For any $\bart>0$ and any $0<t\le\bart$, we have $\omega(t)\ge \left(\omega(\bart)/\bart^2 \right)t^2$. 
\end{lemma}

\begin{proof}
It suffices to show the function $\zeta(t):= \omega(t)/t^2$ is non-increasing on $(0,+\infty)$, or equivalently, that $\zeta'(t) \le 0$ for all $t>0$. Indeed,  we  have $\zeta'(t) = \xi(t)/t^3$, where $\xi(t):= t^2/(1+t) - 2\omega(t)$ for $t>-1$. Since $\xi(0)=0$ and $\xi'(t) = -t^2/(1+t)^2<0$ for all $t >0$, we see that $\xi(t)<0$ for all $t >0$, and hence $\zeta'(t) < 0$ for all $t>0$. This completes the proof. 
%The proof essentially shows that $\omega$ is a strictly decreasing function on $(0,+\infty)$. For details, we refer readers to~\cite{Zhao_21}. 
\end{proof}

We also observe from \eqref{runtime} that $\calX^*$ is the solution to the system of linear equalities and inequalities given therein, and hence there is a %(positive finite) 
Hoffman constant $C_\rmH$ associated with $\calX^*$, namely:

\begin{lemma}\label{lem:Hoffman}
There exists %a finite positive constant 
$0<C_\rmH<+\infty$ such that for all $x\in\Delta_n$, we have
\begin{equation}
\dist_{\normt{\cdot}}(x,\calX^*) \le C_\rmH \normt{Ax - y^*}_{y^*} \ . 
\end{equation}
\end{lemma}

\begin{proof}
This follows from Hoffman's lemma~\cite{Hoffman_52}, whereby there exists $0<C_\rmH<+\infty$ such that %for all $x\in\bbR^n$
\begin{equation}
\dist_{\norm{\cdot}}(x,\calX^*) \le C_\rmH(\normt{(x)_-}_2+\abst{e^\top x-1}+\normt{Ax - y^*}_{y^*}) \ , \quad\forall\, x\in\bbR^n \ ,
\end{equation}
where $(x)_-:= (\min\{0,x_i\})_{i=1}^n$. However, since $x\in\Delta_n$, we have $(x)_- =0$ and $e^\top x-1=0$, and this completes the proof. 
\end{proof}

\begin{remark}
Note that the Hoffman constant $C_\rmH$ depends on both the problem data (namely $A$ and $\{p_j\}_{j=1}^m$)  and the norm $\normt{\cdot}$ on $\bbR^m$. However, for notational brevity, we omit such dependence. 
\end{remark}

Equipped with the above three lemmas, we now present our error bound for $f$ on $\Delta_n$. 

\begin{prop}[An error bound for $f$ on $\Delta_n$] \label{prop:EB}
%Define the univariate function $\omega(s):= s - \ln(1+s)$ for $s>-1$, and $\mu(s):= \omega(s)/s^2$ for $s> 0$. 
% Then we have that   Then  exists a constant $0<\pi(A,y^*)<+\infty$ such that 
 For all $x\in\Delta_n$ we have
\begin{equation}
\dist_{\normt{\cdot}}(x,\calX^*) \le \frac{\sqrt{\ump \, \omega\big( p_{\min}^{-1/2}R_{y^*} \big)}}{C_\rmH R_{y^*}}\sqrt{f^*- f(x)}  \ , 
\end{equation}where $R_{y^*}$, $\omega(\cdot)$, and $C_\rmH$ are defined in \eqref{eq:def_R_y^*}, Lemma~\ref{lem:EB_y}, and Lemma ~\ref{lem:Hoffman}, respectively. 
\end{prop}

\begin{proof}%[Proof of Proposition~\ref{prop:EB}]
Since $f(x)= \barf(A x)$ for all $x\in\Delta_n$, we have %that for all $x\in\Delta_n$:
\begin{align*}
f^* &\ge f(x) + \ump\, \omega(p_{\min}^{-1/2}\normt{Ax-y^*}_{y^*})\\
& \ge f(x) + \ump \frac{\omega\big( p_{\min}^{-1/2}R_{y^*}\big)}{p_{\min}^{-1}R_{y^*}^2}p_{\min}^{-1}\normt{Ax-y^*}_{y^*}^2\\ %\ump\, \omega(\normt{Ax-y^*}_{y^*}), 
%& = f(x) + \ump \frac{\omega\left( \tfrac{1}{\sqrt{\ump}}R_{y^*}\right)}{R_{y^*}^2}\normt{Ax-y^*}_{y^*}^2\\ 
&\ge f(x) + \ump \frac{\omega\big( p_{\min}^{-1/2}R_{y^*}\big)}{R_{y^*}^2C_\rmH^2}\dist_{\normt{\cdot}}(x,\calX^*)^2 \ ,
\end{align*}
where the first inequality uses Lemma~\ref{lem:EB_y}, the second inequality uses Lemma~\ref{lem:omega} and the definition of $R_{y^*}$ in~\eqref{eq:def_R_y^*}, and the third inequality uses Lemma~\ref{lem:Hoffman}.  The proof is completed by rearranging and taking square roots. \end{proof}

By combining Theorem~\ref{thm:O(1/t)} %Corollary~\ref{cor:data_indep_rate} 
and Proposition~\ref{prop:EB}, we arrive at the following theorem, which shows that both $\{\dist_{\norm{\cdot}}(x^t,\calX^*)\}_{t\ge 0}$ and $\{\dist_{\norm{\cdot}}(\barx^t,\calX^*)\}_{t\ge 0}$ converges to zero at $O(1/\sqrt{t})$ rate. 

\begin{theorem} \label{thm:dist}
Define $\dist_{\rm KL}(x^0, \calX^*) := \min_{x\in\calX^*} \,D_{\rm KL}(x,x^0)$ as the ``KL-distance'' from $x^0$ to $\calX^*$. Then for all $t\ge 0$, we have
\begin{align*}
\max\{\dist_{\norm{\cdot}}(x^t,\calX^*),\dist_{\norm{\cdot}}(\barx^t,\calX^*)\}&\le \frac{\sqrt{\ump \, \omega\big( p_{\min}^{-1/2}R_{y^*} \big) \dist_{\rm KL}(x^0, \calX^*)}}{C_\rmH R_{y^*}\sqrt{t+1}} \ ,  
%\quad \mbox{and}\\
%\dist_{\norm{\cdot}}(\barx^t,\calX^*)&\le \frac{\sqrt{\ump \, \omega\big( p_{\min}^{-1/2}R_{y^*} \big) \dist_{\rm KL}(x^0, \calX^*)}}{C_\rmH R_{y^*}\sqrt{t+1}} 
\end{align*}
where $R_{y^*}$, $\omega(\cdot)$, and $C_\rmH$ are defined in \eqref{eq:def_R_y^*}, Lemma~\ref{lem:EB_y}, and Lemma ~\ref{lem:Hoffman}, respectively.  \qed
\end{theorem}

\section{Complexity Comparison of MG, RSGM and FW-LHB} \label{sec:comp}

In this section, we compare the computational complexities of the MG method in~\eqref{eq:MG} with two other {\em principled} first-order methods for finding an $\varepsilon$-optimal solution of~\eqref{eq:P}, which include RSGM and FW-LHB. % Specifically,   in terms of the computational complexity for finding ~\eqref{eq:P}, namely,
Here an $\varepsilon$-optimal solution of~\eqref{eq:P} refers to a feasible point $x\in\Delta_n$ such that $f^* - f(x)\le \varepsilon$. For simplicity, we presume that $p_1=\cdots=p_m = 1/m$ in~\eqref{eq:P}, which then becomes
\begin{equation}
f^*:={\max}_{x\in\Delta_n}\; \left\{f(x):=(1/m)\textstyle\sum_{j=1}^m \ln(a_j^\top x)\right\}. \tag{$\rm P_u$} \label{eq:P_u}
\end{equation}
In addition, for comparison purpose, we presume that all the three methods share the same starting point $x^0 = (1/n)e$, which is a common choice for each of the methods (see e.g.,~\cite{Vardi_85,Kachiyan_96,Lu_18}).

%\subsection{Computational Complexities of RSGM and FW-LHB}

Let us now derive the computational complexities of RSGM and FW-LHB. Before doing so, we first make the following simple but important observation. 

\begin{lemma}\label{lem:bound_init_gap}
In~\eqref{eq:P_u}, for any problem data $\{a_j\}_{j=1}^m\subseteq\bbR_+^n\setminus\{0\}$, we always have 
\begin{equation}
\delta(x^0) = f^* - f(x^0) = f^* - f((1/n)e) \le \ln(n).  \label{eq:init_gap}
\end{equation} 
%data matrix $A\in\bbR_+^{}$
\end{lemma}
\begin{proof}
Let $x^*\in\Delta_n$ be any optimal solution of~\eqref{eq:P_u}. Then we have
\begin{align*}
f^* - f((1/n)e) = \frac{1}{m}\sum_{j=1}^m \ln\left(\frac{a_j^\top x^*}{a_j^\top e/n}\right)\lea \frac{1}{m}\sum_{j=1}^m \ln\left(\frac{\max_{i\in[n]}\,a_{ji} }{\sum_{i=1}^n a_{ji}/n}\right) \leb \ln(n),
%\le \frac{1}{m}\sum_{j=1}^m \ln\left(\frac{\sum_{i=1}^n a_{ji}}{\sum_{i=1}^n a_{ji}/n}\right)
\end{align*}
where in (a) we use $x^*\in\Delta_n$ and in (b) we use $\max_{i\in[n]}\,a_{ji}\le \sum_{i=1}^n a_{ji}$ (for all $j\in[m]$). 
\end{proof}

\begin{remark}
Note that the estimate in~\eqref{eq:init_gap} is indeed tight. To see this, let $a_1=\ldots=a_m=e_1$. Then $f^* = 0$ and $f(x^0) = f((1/n)e) = -\ln(n)$, and hence $\delta(x^0) = \ln(n)$.   
\end{remark}

Next, let us briefly describe the specific form of RSGM when applied to solving~\eqref{eq:P_u}. We first define the ``reference'' function 
\begin{equation}
r(x) := -\textstyle\sum_{i=1}^n \ln (x_i) \quad \mbox{with} \quad  \dom r= \bbR^n_{++}, 
\end{equation}
and its induced Bregman divergence 
\begin{equation}
D_r(y,x):= r(y) - r(x) - \ipt{\nabla r(x)}{y-x}, \quad \forall\,y,x\in \bbR^n_{++}. 
\end{equation}
Starting from $x^0=(1/n) e$, each iteration of RSGM solves the following (Bregman) projection problem:
\begin{equation}
x^{t+1} :=  {\argmax}_{x\in\Delta_n}\; \lranglet{\nabla f(x^t)}{x} - D_r(x,x^t). %, \quad \forall\, t\ge 0. 
\tag{BP} \label{eq:BP}
\end{equation}
 The computational complexity of RSGM for finding an $\varepsilon$-optimal solution of~\eqref{eq:P_u} is stated below. %as follows.

\begin{prop}
Starting from $x^0=(1/n) e$, RSGM finds an $\varepsilon$-optimal solution of~\eqref{eq:P_u} in at most
\begin{align*}
O\left(\frac{(mn+n\ln(n))n}{\varepsilon}\ln\left(\frac{\ln(n)}{\varepsilon}\right)\right)\quad \mbox{arithmetic operations}. 
\end{align*}
% $$ arithmetic operations. 
%\begin{equation}
%O((n/\varepsilon)\ln(\delta(x^0)/\varepsilon))
%\end{equation}
\end{prop}
\begin{proof}
Since $f$ is $1$-smooth relative to $r$ on $\bbR^n_{++}$ (cf.~\cite[Lemma~7]{Bauschke_17}), from~\cite[Theorem~3.1]{Lu_18} (see also~\cite[Theorem 1(iv)]{Bauschke_17}), we know that for all $x\in\ri\Delta_n$, 
\begin{equation}
f(x) - f(x^t)\le D_r(x,x^0)/t, \quad \forall\,t\ge 1. \label{eq:conv_RSGM}
\end{equation}
Since it may happen that $\ri\Delta_n\cap\calL=\emptyset$ and hence $\calX^*\cap\ri\Delta_n=\emptyset$, similar to~\cite[Theorem 4.1]{Lu_18}, let us fix any $x^*\in\calX^*$, and define $\hatx := (1-\alpha)x^* + \alpha x^0$ with $\alpha:= \varepsilon/(2\delta(x^0))$. This ensures that 
\begin{equation}
\delta(\hatx)\le (1-\alpha)\delta(x^*) + \alpha\delta(x^0) = \varepsilon/2. \label{eq:delta_hatx}
\end{equation} 
In addition, since $x^0=(1/n)e$, we have $\nabla r(x^0) = -n e$ and hence %$$we can bound $D_r(\hatx,x^0)$ as follows:
\begin{equation}
D_r(\hatx,x^0) = r(\hatx) - r(x^0) \le r(\alpha x^0) - r(x^0) = -n\ln(\alpha) = n\ln(2\delta(x^0)/\varepsilon), \label{eq:bound_D_r}
\end{equation}
where the inequality follows from that $\hatx\ge \alpha x^0$ and the function $r$ is coordinate-wise decreasing.  From~\eqref{eq:delta_hatx}, it is clear that if $f(\hatx) - f(x^t)\le \varepsilon/2$, then $\delta(x^t)\le \varepsilon$, and from~\eqref{eq:bound_D_r}, we know that this happens in no more than 
\begin{equation}
\lceil 2(n/\varepsilon)\ln(2\delta(x^0)/\varepsilon)\rceil = O((n/\varepsilon)\ln(\delta(x^0)/\varepsilon)) = O((n/\varepsilon)\ln(\ln(n)/\varepsilon)) \quad {\rm iterations}.  \label{eq:T_RSGM}
\end{equation}
Next, we analyze the complexity of arithmetic operations in each iteration. % of RSGM. 
%In addition, in each iteration of RSGM, 
First, note that computing $\nabla f(x^t)$ requires $O(mn)$ arithmetic operations, for all $t\ge 0$. In addition, from~\cite[Section~7]{Nest_11}, % (), 
we know that the projection problem in~\eqref{eq:BP} be reduced to finding the unique root of a strictly decreasing univariate function on $(0,+\infty)$, and the root can  be computed to machine precision in $O(n\ln(n))$ arithmetic operations (see also~\cite{Ye_92,Bell_08}). Therefore, each iteration of RSGM requires $O(mn+n\ln(n))$ arithmetic operations. This, together with~\eqref{eq:T_RSGM}, completes the proof. 
%Note that the  result in~\eqref{eq:conv_RSGM} does not directly imply the convergence rate of $\delta(x^t) = f^* - f(x^t)$.from the characterization of $\calX^*$ in~\eqref{eq:char_X*}, it is clearly possible that $\ri\Delta_n\cap\calL=\emptyset$ and hence $\calX^*\cap\ri\Delta_n=\emptyset$. Therefore, 
\end{proof}

Now, let us switch our focus to FW-LHB. Note that FW-LHB cannot be directly applied to~\eqref{eq:P}, since 
the objective function $f$ in~\eqref{eq:P_u} is not self-concordant (cf.~\cite[Section~2.1]{Nest_94}). Instead, we apply FW-LHB to an equivalent problem of~\eqref{eq:P_u}, namely
\begin{equation}
{\min}_{x\in\Delta_n}\; -\textstyle\sum_{j=1}^m \ln(a_j^\top x).\tag{$\mathrm{P_s}$} \label{eq:P_s}
\end{equation}
Clearly, $x\in\Delta_n$ is an $\varepsilon$-optimal solution of~\eqref{eq:P_u} if and only if it is an $\bar{\varepsilon}$-optimal solution of~\eqref{eq:P_s} for $\bareps = m\varepsilon$. Based on this, we state the computational complexity of FW-LHB for finding an $\varepsilon$-optimal solution of~\eqref{eq:P_u} as follows.

\begin{prop}
Starting from $x^0=(1/n) e$, FW-LHB finds an $\varepsilon$-optimal solution of~\eqref{eq:P_u} in at most
\begin{align*}
%O\big( m \delta(x^0)\ln(m \delta(x^0)) + m^2/\bareps\big) = 
O\big( m^2n \ln(n)(\ln(m)+\ln \ln(n)) + m^2n/\varepsilon\big) \quad \mbox{arithmetic operations}.%\label{eq:comp_FW}
\end{align*}
\end{prop}

\begin{proof}
From~\cite[Remark~2.1]{Zhao_20b}, we know that to find an $\bareps$-optimal solution of~\eqref{eq:P_s}, the iteration complexity of FW-LHB is 
\begin{equation}
O\big( m \delta(x^0)\ln(m \delta(x^0)) + m^2/\bareps\big) = O\big( m \delta(x^0)\ln(m \delta(x^0)) + m/\varepsilon\big).%\label{eq:comp_FW}
\end{equation}
Since $\delta(x^0)\le \ln(n)$ (cf.\ Lemma~\ref{lem:bound_init_gap}) and each iteration of FW-LHB requires $O(mn)$ arithmetic operations (cf.~\cite[Section~2]{Zhao_20b}), we complete the proof. 
\end{proof}

Finally, from Theorem~\ref{thm:O(1/t)}, we can easily derive the computational complexity of MG for finding an $\varepsilon$-optimal solution of~\eqref{eq:P_u} as follows: 

\begin{prop}
Starting from $x^0=(1/n) e$, MG finds an $\varepsilon$-optimal solution of~\eqref{eq:P_u} in at most
\begin{align*}
O\big( mn\ln(n)/\varepsilon\big) \quad \mbox{arithmetic operations}.%\label{eq:comp_FW}
\end{align*}
\end{prop}

\begin{table}\centering\renewcommand{\arraystretch}{1.5}
\caption{The computational complexities of MG, RSGM and FW-LHB for finding an $\varepsilon$-optimal solution of~\eqref{eq:P_u}.} \label{tab:comp}
\begin{tabular}{|c|c|}\hline
RSGM &  $O\left(\frac{(mn+n\ln(n))n}{\varepsilon}\ln\left(\frac{\ln(n)}{\varepsilon}\right)\right)$ \\\hline
FW-LHB & $O\left( m^2n \ln(n)(\ln(m)+\ln \ln(n)) + \frac{m^2n}{\varepsilon}\right)$ \\\hline
MG &  $O\left( \frac{mn\ln(n)}{\varepsilon}\right)$  \\\hline
\end{tabular}
\end{table}

For ease of comparison, we summarize the computational complexities of MG, RSGM and FW-LHB for finding an $\varepsilon$-optimal solution of~\eqref{eq:P_u} in Table~\ref{tab:comp}. From this table, we see that MG always has a lower complexity compared to RSGM. In addition, for sufficiently small accuracy $\varepsilon$ --- so that the computational complexity of FW-LHB becomes (essentially) $O(m^2n/\varepsilon)$, MG has a lower complexity compared to FW-LHB as long as $n = O(\exp(m))$. This %explains (in theory) 
provides an explanation to the superior numerical performance of MG compared to RSGM and FW-LHB on the PET problem as observed in~\cite[Section~4.2]{Zhao_20b}.

\section*{Acknowledgment}

%{\bf Acknowledgment.} 
The author's research is supported by AFOSR Grant No. FA9550-19-1-0240. The author is grateful for the helpful comments from Prof.\ R.\ M.\ Freund that greatly improve the exposition of the current manuscript. 

\appendix

\section{Proof of Lemma~\ref{lem:monotonicity}}\label{app:proof}

The following proof is borrowed from Cover~{\cite[Theorem~1]{Cover_84}}. 
We have 
\begin{align*}
f(x^{t+1}) - f(x^t) &= \sum_{j=1}^m p_j \ln\left(\frac{a_j^\top x^{t+1}}{a_j^\top x^{t}}\right)\\
 &= \sum_{j=1}^m p_j \ln\left(\sum_{i=1}^n\frac{a_{ji} x_i^t}{a_j^\top x^{t}}\frac{x_i^{t+1}}{x_i^t}\right)\\
&\ge \sum_{j=1}^m p_j \sum_{i=1}^n\frac{a_{ji} x_i^t}{a_j^\top x^{t}} \ln\left(\frac{x_i^{t+1}}{x_i^t}\right)\\
&= \sum_{i=1}^n x_i^t \nabla_i f(x^t) \ln\left(\frac{x_i^{t+1}}{x_i^t}\right)\\
& = \sum_{i=1}^n x_i^{t+1} \ln\left(\frac{x_i^{t+1}}{x_i^t}\right) = D_{\rm KL}(x^{t+1},x^t) \ , 
\end{align*}
where the inequality follows from the concavity of $\ln(\cdot)$, and the last equality uses~\eqref{eq:def_KL}. \qed

%\acknowledgment

%\begin{acknowledgment}
%\end{acknowledgment}

%\end{proof}
\small
\bibliographystyle{IEEEtr}
\bibliography{math_opt,mach_learn,stat_ref}

\end{document}